\documentclass[a4paper,11pt]{article}

\usepackage{amsmath,amsthm,amssymb,fullpage,tikz}
\usetikzlibrary{arrows.meta}

\newcommand{\meg}[1]{\operatorname{meg}(#1)}
\newcommand{\xmeg}[1]{\overline{\operatorname{meg}}(#1)}
\newcommand{\NP}{\textsf{NP}}
\newcommand{\boxprod}{\mathbin{\square}}
\newcommand{\dist}{\operatorname{dist}}
\newcommand{\diam}{\operatorname{diam}}
\newcommand{\ie}{i.e.\ }

\newtheorem{theorem}{Theorem}
\newtheorem{lemma}[theorem]{Lemma}
\newtheorem{cor}[theorem]{Corollary}
\newtheorem{prop}[theorem]{Proposition}

\theoremstyle{remark}
\newtheorem*{remark}{Remark}

\title{Monitoring edge-geodetic sets: hardness and graph products}
\author{John Haslegrave\thanks{Mathematical Institute, University of Oxford. E-mail: \texttt{j.haslegrave@cantab.net}.}}

\begin{document}
\maketitle

\begin{abstract}
Foucaud, Krishna and Ramasubramony Sulochana recently introduced the concept of monitoring edge-geodetic sets in graphs, and a related graph invariant. These are sets of vertices such that the removal of any edge changes the distance between some pair of vertices in the set. They studied the minimum possible size of such a set in a given graph, which we call the monitoring edge-geodetic number.

We show that the decision problem for the monitoring edge-geodetic number is \NP-complete. We also give best-possible upper and lower bounds for the Cartesian and strong products of two graphs. These bounds establish the exact value in many cases, including many new examples of graphs whose only monitoring edge-geodetic set is the whole vertex set.

\end{abstract}

\section{Introduction}
Network monitoring covers a range of applications where a network (which we think of here as a finite simple graph) is susceptible to sporadic failures, either by the loss of nodes or of connections, which it is necessary to observe or respond to. Often the relevant effects of these failures are to temporarily increase distances or even to separate some nodes from the remainder of the network. Of course, if we could continuously observe the entire graph then we would immediately become aware of any failures, but it is desirable to find a way to achieve the same goal using less information.

Very recently, Foucaud, Krishna and Ramasubramony Sulochana \cite{FKL22+} introduced a very natural parameter for network monitoring by distances between a subset of vertices. We say that a set of vertices \textit{monitors} an edge of the graph if removing that edge would change the distance between some pair of vertices in the set. A \textit{monitoring edge-geodetic set}, or \textit{MEG-set} for short, is a set that monitors all edges of the graph. The problem is to find a MEG-set that is as small as possible. The minimum size of a MEG-set of $G$ is the \textit{monitoring edge-geodetic number}, written $\meg{G}$. The number of vertices of degree $1$, $L(G)$, is clearly a lower bound on $\meg{G}$, since the edge incident to a vertex of degree $1$ can only be monitored by a set if that vertex is in the set. Foucaud, Krishna and Ramasubramony Sulochana show that this bound is attained by any forest, and in general prove an upper bound on $\meg{G}-L(G)$ in terms of the cycle rank of $G$. They also give the exact value of $\meg{G}$ in several other cases.

This setting shares features with several others which appear in the literature, and may indeed be regarded as a combination of them. It is closely related to \textit{distance edge-monitoring} \cite{FKMR,FKKMR}, where instead we ask for the weaker property that the removal of any edge increases the distance between some pair of vertices, at least one of which is in the set. In a different direction, it is related to the concept of \textit{edge-geodetic sets} \cite{SJ07}, which are sets of vertices such that any edge in the graph lies on a shortest path between some pair of vertices in the set. Since the property of being a MEG-set may be rephrased as requiring any edge in the graph to lie on every shortest path between some pair of vertices in the set, every MEG-set is both an edge-geodetic set and a distance edge-monitoring set (although the converse is not true). Edge-geodetic sets arose as a natural variation of \textit{geodetic sets} \cite{HLT93}, where we require every vertex to lie on a shortest path between some pair of vertices in the set. Further variations include \textit{strong edge-geodetic sets} \cite{MKXAT}, where we can choose a specific shortest path for each pair of vertices in the set such that the union of these paths is the entire edge set. These, and similar problems, arise naturally in applications such as choosing a patrolling scheme for road inspectors \cite{MKXAT} or using path-oriented tools (such as ping) to monitor IP networks \cite{BR03}.

In Section \ref{sec:prod} we give a general approach to Cartesian products of graphs, giving upper and lower bounds for the monitoring edge-geodetic number of the Cartesian product of $G$ and $H$ in terms of $\meg{G}$ and $\meg{H}$. In cases where the two bounds differ, we show that either bound can give the correct value. Our bounds also apply to strong products; we also show that the monitoring edge-geodetic number of the strong product is at least that of the Cartesian product, and that this inequality can be strict. As a corollary we deduce exact values for general graphs arising as products of paths, which include two results of \cite{FKL22+} as special cases. Particular attention is paid in \cite{FKL22+} to graphs satisfying $\meg{G}=|G|$, and we give many new examples by showing that any (Cartesian or strong) product of such a graph with an arbitrary second graph also has this property. 

In Section \ref{sec:np} we address the other question raised in \cite{FKL22+}, that of computational complexity. We show that determining the minimum size of a MEG-set is \NP-hard, that is, the decision problem to determine, for a given graph $G$ and integer $k$, whether $\meg{G}\leq k$ is \NP-complete.

We use standard graph-theoretic notation, and in particular write $\dist_G(u,v)$ for the graph distance in $G$ between vertices $u$ and $v$, and $\diam(G)$ for the diameter of $G$, that is the maximum value of $\dist_G(u,v)$ over all pairs of vertices.

\section{Bounds for graph products}\label{sec:prod}
We consider two standard ways to form a product of graphs $G$ and $H$: the Cartesian product $G\boxprod H$ and the strong product $G\boxtimes H$. Both have vertex set $V(G)\times V(H)$. The edges of $G\boxprod H$ consist of pairs $(a,b)(c,d)$ such that either $a=c$ and $bd\in E(H)$ or $b=d$ and $ac\in E(G)$. The edges of $G\boxtimes H$ consist of the edges of $G\boxprod H$ together with the pairs $(a,b)(c,d)$ with $ac\in E(G)$ and $bd\in E(H)$. In each case the symbol used represents the product of a single edge with itself. The Cartesian product of two paths is a grid, whereas the strong product of two paths is sometimes called a ``king's graph'', since edges correspond to possible moves of a chess king. Using cycles instead of paths gives toroidal grids and toroidal king's graphs respectively.

The main result of this section is the following sequence of bounds.
\begin{theorem}\label{thm:bounds}
For any two graphs $G$ and $H$, we have
\begin{align*}\max\{\meg{G}|H|,\meg{H}|G|\}&\leq\meg{G\boxprod H}\\
&\leq\meg{G\boxtimes H}\\
&\leq\meg{G}|H|+\meg{H}|G|-\meg{G}\meg{H}.\end{align*}
Furthermore, if either $G$ or $H$ has a unique minimal MEG-set, then we have
\[\meg{G\boxprod H}=\meg{G\boxtimes H}=\meg{G}|H|+\meg{H}|G|-\meg{G}\meg{H},\]
and if both $G$ and $H$ have unique minimal MEG-sets, then so do $G\boxprod H$ and $G\boxtimes H$.
\end{theorem}
\begin{remark}Given the form of these bounds, it would perhaps be more natural to think about the number of vertices \textit{not} in a minimum MEG-set. Writing $\xmeg{G}$ for the maximum order of the complement of a MEG-set, we can equivalently express the bounds above as
\[\min\{\xmeg{G}|H|,\xmeg{H}|G|\}\geq\xmeg{G\boxprod H}\geq\xmeg{G\boxtimes H}\geq\xmeg{G}\xmeg{H}.\]
\end{remark}
\begin{proof}To show the first inequality, we classify the MEG-sets of $G\boxprod H$. Given a set $S\subseteq V(G)\times V(H)$, and vertices $x\in V(G)$ and $y\in V(H)$, we write $S_{x,*}$ for $\{v\in V(H):(x,v)\in S\}$ and $S_{*,y}$ for $\{v\in V(G):(v,y)\in S\}$.

Suppose $(a,b)$ and $(c,d)$ are vertices with $a\neq c$ and $b\neq d$. Then there are two edge-disjoint (in fact, internally vertex-disjoint) shortest paths from $(a,b)$ to $(c,d)$ in $G\boxprod H$: one via $(c,b)$ where all intermediate vertices are of the form $(x,b)$ or $(c,y)$, and one via $(a,d)$ where all intermediate vertices are of the form $(x,d)$ or $(a,y)$. Conversely, all shortest paths from $(a,b)$ to $(c,b)$ are of the form $(a_0,b)\cdots(a_r,b)$, where $a=a_0\cdots a_r=c$ is a shortest path in $G$.

Consequently, an edge of the form $(x,y)(z,y)$ lies on all shortest paths between vertices $(a,b)$ and $(c,d)$ if and only if $b=d=y$, and $xz$ lies on all shortest paths between $a$ and $c$ in $G$. It follows that $S$ monitors all edges of the form $(x,y)(z,y)$ if and only if $S_{*,y}$ is a MEG-set of $G$. Similar arguments apply to edges of the form $(x,y)(x,z)$, and so $S$ is a MEG-set of $G\boxprod H$ if and only if $S_{x,*}$ is a MEG-set of $H$ for every $x\in V(G)$ and $S_{*,y}$ is a MEG-set of $G$ for every $y\in V(H)$.

Thus for any MEG-set $S$ we have $|S|=\sum_{x\in V(G)}|S_{x,*}|\geq |G|\meg{H}$, and likewise we have $|S|=\sum_{y\in V(H)}|S_{*,y}|\geq |H|\meg{G}$, as required.

\medskip
Next we show that $\meg{G\boxprod H}\leq\meg{G\boxtimes H}$. Let $S$ be any MEG-set of $G\boxtimes H$; it suffices to show that $S_{x,*}$ must be a MEG-set of $H$ for every $x\in V(G)$ and $S_{*,y}$ must be a MEG-set of $G$ for every $y\in V(H)$, since we showed earlier that this implies $S$ is a MEG-set of $G\boxprod H$. Suppose $(a,b)$ and $(c,d)$ are vertices with $a\neq c$ and $b\neq d$, and choose shortest paths $a=a_0\cdots a_r=c$ in $G$ and $b=b_0\cdots b_s=d$ in $H$. If $r=s$ then $(a_0,b_0)\cdots(a_r,b_r)$ is a shortest path from $(a,b)$ to $(c,d)$ in $G\boxtimes H$, and uses no edges of the form $(x,y)(x,z)$ or $(x,y)(z,y)$. If $r>s$ then both $(a_0,b_0)\cdots(a_s,b_s)\cdots(a_r,b_s)$ and $(a_0,b_0)\cdots(a_{r-s},b_0)\cdots(a_r,b_s)$ are shortest paths from $(a,b)$ to $(c,d)$ in $G\boxtimes H$, and have no edges of the form $(x,y)(x,z)$ or $(x,y)(z,y)$ in common, and similarly if $r<s$ then we can construct two shortest paths with this property. Thus in $G\boxtimes H$, an edge of the form $(x,y)(z,y)$ lies on all shortest paths between vertices $(a,b)$ and $(c,d)$ if and only if $b=d=y$, and $xz$ lies on all shortest paths between $a$ and $c$ in $G$, and a similar categorisation applies to edges of the form $(x,y)(x,z)$. As before, it follows that  $S_{x,*}$ must be a MEG-set of $H$ for every $x\in V(G)$ and $S_{*,y}$ must be a MEG-set of $G$ for every $y\in V(H)$.

\medskip
We now prove the upper bound. Choose a MEG-set $S$ for $G$ of order $\meg{G}$ and a MEG-set $T$ for $H$ of order $\meg{H}$, and consider the set  $S\vee T=\{(a,b)\in V(G)\times V(H):a\in S\text{ or }b\in T\}$. We have $|S\vee T|=|S||H|+|G||T|-|S||T|=\meg{G}|H|+\meg{H}|G|-\meg{G}\meg{H}$, so it suffices to prove that $S\vee T$ is a MEG-set for $G\boxtimes H$.

Let $e$ be an edge of $G\boxtimes H$. If $e=(a,b)(c,b)$ then $ac\in E(G)$ and so there are vertices $x,y\in S$ such that $ac$ lies on all shortest paths from $x$ to $y$ in $G$. Consequently $e$ lies on all shortest paths from $(x,b)$ to $(y,b)$ in $G\boxtimes H$, and so $S\vee T$ monitors $e$. A similar argument applies in the case $e=(a,b)(a,d)$. The only remaining case is if $e=(a,b)(c,d)$ with $ac\in E(G)$ and $bd\in E(H)$. Now there are vertices $a',c'\in S$ and $b',d'\in T$ such that $ac$ lies on all shortest paths from $a'$ to $c'$ in $G$ and $bd$ lies on all shortest paths from $b'$ to $d'$ in $G$. Note that the shortest paths from $a'$ to $c'$ in $G$ are precisely the paths obtained by concatenating a shortest path from $a'$ to $a$, the edge $ac$, and a shortest path from $c$ to $c'$ (where a path may have length $0$, if $a'=a$ or $c'=c$), and likewise for $b'$-$d'$ paths in $H$. If $\dist_G(a',a)\leq\dist_H(b',b)$, set $a''=a'$ and choose a vertex $b''$ on a shortest path from $b'$ to $b$ in $H$ such that $\dist_H(b'',b)=\dist_G(a'',a)$; otherwise set $b''=b'$ and choose a vertex $a''$ on a shortest path from $a'$ to $a$ in $G$ such that $\dist_G(a'',a)=\dist_H(b'',b)$. Note that, since either $a''=a'\in S$ or $b''=b'\in T$, we have $(a'',b'')\in S\vee T$. Similarly we may choose $c''$ and $d''$ such that $c''$ lies on a shortest path from $c$ to $c'$ in $G$ and $d''$ lies on a shortest path from $d$ to $d'$ in $H$, satisfying $\dist_G(c,c'')=\dist_H(d,d'')$ and $(c'',d'')\in S\vee T$. Now every shortest path from $a''$ to $c''$ in $G$ must include $ac$ (since otherwise we could extend it to a path from $a'$ to $c'$ avoiding $ac$ of length $\dist_G(a',c')$), and every shortest path from $b''$ to $d''$ in $H$ must include $bd$; furthermore, by choice of endpoints, any pair of such paths must have the same length, and include $ac$ and $bd$ in the same position. Since $\dist_G(a'',c'')=\dist_H(b'',d'')$, the shortest paths from $(a'',c'')$ to $(b'',d'')$ in $G\boxtimes H$ are paths of the form $(a_0,b_0)\cdots(a_r,b_r)$, where $a''=a_0\cdots a_r=c''$ and $b''=b_0\cdots b_r=d''$ are shortest paths, thus all contain $(a,b)(c,d)$. Thus $S\vee T$ monitors $e$, as required.

\medskip
Finally, we show that $\meg{G\boxprod H}$ equals the upper bound when $G$ has a unique minimal MEG-set, and that $G\boxprod H$ has a unique minimal MEG-set whenever both $G$ and $H$ do. Since we already showed that any MEG-set of $G\boxtimes H$ is a MEG-set of $G\boxprod H$, the corresponding facts for $G\boxtimes H$ will follow.

Suppose that $G$ has unique minimal MEG-set $S$, and let $R$ be a MEG-set of $G\boxprod H$. Then any MEG-set of $G$ contains $S$, and, since $R_{*,y}$ is a MEG-set for every $y$, we have $(x,y)\in R$ for every $x\in S$ and $y\in V(H)$, \ie $R_{x,*}=V(H)$ for every $x\in S$. Also, since $R_{x,*}$ is a MEG-set of $H$ for every $x\not\in S$, we have 
\[|R|=\sum_{x\in V(G)}|R_{x,*}|\geq |S||H|+(|G|-|S|)\meg{H}=\meg{G}|H|+|G|\meg{H}-\meg{G}\meg{H},\]
as required. If additionally $H$ has a unique minimal MEG-set $T$ then $T\subseteq R_{x,*}$ for each $x\in V(G)$, and so $S\vee T\subseteq R$. But if $R$ is a minimal MEG-set then $|R|\leq|S\vee T|$, so we must have $R=S\vee T$.
\end{proof}

We immediately obtain the following fact, giving a wide variety of new graphs $F$ with $\xmeg{F}=0$, \ie $\meg{F}=|F|$.
\begin{cor}
Let $G$ be any graph with $\xmeg{G}=0$ and $H$ be any graph. Then $\xmeg{G\boxprod H}=\xmeg{G\boxtimes H}=0$.
\end{cor}
In particular, any graph which is a Cartesian or strong product of any graph with $K_r$, where $r\geq 2$, has this property. Using \cite[Theorem 11]{FKL22+}, we may also take $G$ to be any complete multipartite graph other than a star.

We also obtain, using the last part of Theorem \ref{thm:bounds} iteratively, the exact monitoring edge-geodetic number for any product of any number of paths; this includes as special cases \cite[Theorem 12, Theorem 13]{FKL22+} (corresponding to $m_1=\cdots=m_k=2$ and $k=2$ respectively).
\begin{cor}\label{cor:grids}Let $k\geq 1$ be an integer and $m_1,\ldots,m_k\geq 2$. Let the graph $G$ be obtained as the Cartesian or strong product of the paths $P_{m_1},\ldots, P_{m_k}$ (or as a product of these paths using an arbitrary combination of Cartesian and strong products). Then 
\[\meg{G}=\prod_{i=1}^km_i-\prod_{i=1}^k(m_i-2);\] furthermore $G$ has a unique minimal MEG-set.
\end{cor}
\begin{proof}
We use induction on $k$. For $k=1$, we have a path of length at least $1$, and a set is a MEG-set if and only if it includes both endpoints. For $k>1$, we may write $G=G_1\boxprod G_2$ or $G=G_1\boxtimes G_2$, where $G_1$ is a product of $P_{m_1},\ldots,P_{m_{k'}}$ and $G_2$ is a product of $P_{m_{k'}+1},\ldots,P_{m_k}$, and $1<k'<k$. By the induction hypothesis $\xmeg{G_1}=\prod_{i=1}^{k'}(m_i-2)$ and $\xmeg{G_2}=\prod_{i=k'+1}^k(m_i-2)$, and both have unique minimal MEG-sets. Thus Theorem \ref{thm:bounds} gives the required result. 
\end{proof}

We next consider the question of whether our bounds can be improved when they differ. Of course, there are many cases already referred to for which lower and upper bounds coincide, in which case neither bound can be improved. However, it is natural to ask whether when the bounds are different either can be correct. 

Corollary \ref{cor:grids} shows that the upper bound is attained by an arbitrary product of paths. Provided we have $m_i\geq 2$ for each $i$, the upper and lower bounds do not coincide in this case. Our next result shows that the discrete torus gives an example where the two bounds differ but the lower bound is attained.

\begin{theorem}For $m\geq 5$ the discrete torus $C_m\boxprod C_m$ has $\meg{C_m\boxprod C_m}=3m$.
\end{theorem}
\begin{proof}Set $V(C_m)=\{0,\ldots, m-1\}$. Choose a set $S\subset\{0,\ldots,m-1\}$ of order $3$ that is a MEG-set for $C_m$ (e.g.\ by taking $S=\{0,\lfloor (m-1)/2\rfloor,2\lfloor (m-1)/2\rfloor\}$). Let $T$ be the set of pairs $(i,j)$ with $0\leq i,j\leq m-1$ such that $i+j\mod m\in S$. Clearly $|T|=3m$. See Figure \ref{fig:torus} for an example. Then for each $i$ we have $T_{*,i}$ and $T_{i,*}$ are obtained from $S$ by a rotation, so are also MEG-sets of $C_m$. From the proof of Theorem \ref{thm:bounds}, this means $T$ is a MEG-set for $C_m\boxprod C_m$. Since $\meg{C_m}=3$ by \cite[Theorem 8]{FKL22+}, this matches the lower bound in Theorem \ref{thm:bounds}. 
\end{proof}
\begin{figure}
\centering
\begin{tikzpicture}[thick]
\foreach \x in {1,...,5}{
\draw (\x,0.5) -- (\x,5.5);
\draw (0.5,\x) -- (5.5,\x);}

\draw[thin,-{[length=2mm] To[]}] (0.5,0.5) -- (0.5,3);
\draw[thin] (0.5,3) -- (0.5,5.5);
\draw[thin,-{[length=2mm] To[] To[]}] (0.5,5.5) -- (3,5.5);
\draw[thin] (3,5.5) -- (5.5,5.5);
\draw[thin,-{[length=2mm] To[]}] (5.5,0.5) -- (5.5,3);
\draw[thin] (5.5,3) -- (5.5,5.5);
\draw[thin,-{[length=2mm] To[] To[]}] (0.5,0.5) -- (3,0.5);
\draw[thin] (3,0.5) -- (5.5,0.5);

\filldraw (1,1) circle (0.1);
\filldraw[fill=white] (1,2) circle (0.1);
\filldraw (1,3) circle (0.1);
\filldraw[fill=white] (1,4) circle (0.1);
\filldraw (1,5) circle (0.1);

\filldraw[fill=white] (2,1) circle (0.1);
\filldraw (2,2) circle (0.1);
\filldraw[fill=white] (2,3) circle (0.1);
\filldraw (2,4) circle (0.1);
\filldraw (2,5) circle (0.1);

\filldraw (3,1) circle (0.1);
\filldraw[fill=white] (3,2) circle (0.1);
\filldraw (3,3) circle (0.1);
\filldraw (3,4) circle (0.1);
\filldraw[fill=white] (3,5) circle (0.1);

\filldraw[fill=white] (4,1) circle (0.1);
\filldraw (4,2) circle (0.1);
\filldraw (4,3) circle (0.1);
\filldraw[fill=white] (4,4) circle (0.1);
\filldraw (4,5) circle (0.1);

\filldraw (5,1) circle (0.1);
\filldraw (5,2) circle (0.1);
\filldraw[fill=white] (5,3) circle (0.1);
\filldraw (5,4) circle (0.1);
\filldraw[fill=white] (5,5) circle (0.1);
\end{tikzpicture}
\caption{The discrete torus $C_5\boxprod C_5$ with a MEG-set of order $15$ (filled vertices).}
\label{fig:torus}
\end{figure}
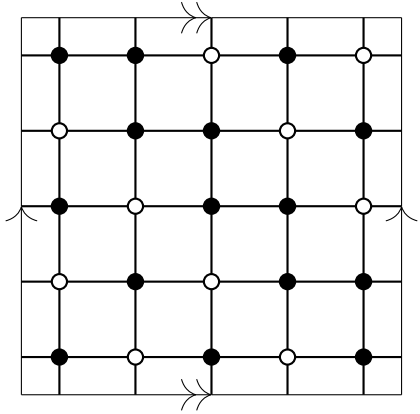

We showed above that $\meg{G\boxprod H}\leq \meg{G\boxtimes H}$. However, we have seen many examples where the two are equal. We next give an example where the inequality is strict. We will use the following result relating $\meg{G}$ to the vertex cover number when $G$ has diameter $2$, which may be of independent interest. The \textit{vertex cover number} of $G$, written $\tau(G)$, is the order of the smallest set of vertices which contains at least one endpoint of every edge. The associated decision problem was one of Karp's original \NP-complete problems \cite{Kar72}.

\begin{lemma}\label{lem:diam}For any graph $G$ with $\diam(G)=2$, any MEG-set is also a vertex cover. In particular, $\meg{G}\geq \tau(G)$.
\end{lemma}
\begin{proof}Let $S$ be a MEG-set, and let $e$ be any edge of $G$. Since $S$ monitors $e$, there is a shortest path between vertices of $S$ that contains $e$. Since all shortest paths have length at most $2$, this path consists of $e$ together with at most one other edge, so at least one endvertex of the path is also an endvertex of $e$. Thus at least one endvertex of $e$ is in $S$; since $e$ was arbitrary, $S$ is a vertex cover.
\end{proof}

We can now find $\meg{C_5\boxtimes C_5}$, which exceeds $\meg{C_5\boxprod C_5}=15$.
\begin{prop}We have $\meg{C_5\boxtimes C_5}=20$.
\end{prop}
\begin{proof}We first show that $\tau(C_5\boxtimes C_5)\geq 20$, which gives the required lower bound by Lemma \ref{lem:diam} since $\diam(G\boxtimes G)=\diam(G)$. Denote the vertices of $C_5\boxtimes C_5$ by $\{(i,j):0\leq i,j\leq 4\}$. Suppose that $S$ is a vertex cover. We claim that $|S_{i,*}\cup S_{i+1,*}|\geq 8$ for each $i\in\{0,\ldots,4\}$, where we take addition modulo $5$. Since $S_{i,*}$ covers the edge $(i,j)(i,j+1)$ for each $j\in\{0,\ldots,4\}$, we must have $|S_{i,*}|\geq 3$, and similarly $|S_{i+1,*}|\geq 3$. If $|S_{i,*}|=3$ then the missing vertices must be non-adjacent, \ie of the form $(i,j)$ and $(i,j+2)$ for some $j\in\{0,\ldots,4\}$. But then in order to cover all edges between these two vertices and $(i+1,0),\ldots,(i+1,4)$, we must have $|S_{i+1,*}|=5$. Similarly, if $|S_{i+1,*}|=3$ then $|S_{i,*}|=5$, and if neither of these apply then $|S_{i,*}|,|S_{i+1,*}|\geq 4$, so the claim holds. Thus we have \[|S|=\sum_{i=0}^4|S_{i,*}|=\frac12\sum_{i=0}^4|S_{i,*}\cup S_{i+1,*}|\geq 20,\]
as required.

Finally, we claim that $V(G)\setminus\{(0,0),(1,2),(2,4),(3,1),(4,3)\}$ is a MEG-set. It clearly monitors all edges not adjacent to one of the missing vertices. Also, the neighbours of a missing vertex $v$ may be divided into four opposing pairs, such that each pair has a unique shortest path, which uses $v$, and therefore the edges between that pair and $v$ are monitored.
\end{proof}

Finally in this section, we observe that the property of having a unique minimal MEG-set in Theorem \ref{thm:bounds} cannot be relaxed to that of having a unique minimum MEG-set. For example, the graph $G$ obtained from the $5$-cycle $abcde$ by adding two pendant edges $aa'$ and $bb'$ has unique minimum MEG-set $S=\{a',b', d\}$. However, the set $T=\{a',b',c,e\}$ is also a minimal MEG-set. In this case the set $(S\vee S)\setminus\{(d,d)\}$ is a MEG-set for $G\boxprod G$, showing that the upper bound in Theorem \ref{thm:bounds} is not attained in this case.
\section{Computational complexity}\label{sec:np}
In this section we prove \NP-hardness of determining the monitoring edge-geodetic number.
\begin{theorem}The decision problem of determining for a graph $G$ and natural number $k$ whether $\meg{G}\leq k$ is \NP-complete.
\end{theorem}
\begin{proof}
This problem is clearly in \NP, since a MEG-set of order $k$ witnesses $\meg{G}\leq k$, and it is straightforward to verify in polynomial time that a given set is a MEG-set by computing all pairwise distances for $G$ and for every graph obtained by removing a single edge.

To show hardness, we give a Karp reduction from SAT, that is, given a SAT instance we construct (in polynomial time) a graph $G$ and integer $k$ such that the instance is satisfiable if and only if $\meg{G}\leq k$. 

Suppose that the instance of SAT consists of $n$ clauses $C_1,\ldots,C_n$ using $m$ variables $a_1,\ldots,a_m$. We may assume that both $a_i$ and $\neg a_i$ appear somewhere for each $i$, since if only one of them appears then we may choose the truth value of $a_i$ appropriately and remove all clauses containing it without affecting satisfiability. We also assume that no clause contains both $a_i$ and $\neg a_i$ for any $i$ (since otherwise we may remove that clause without affecting satisfiability), and that $n\geq 2$ (since otherwise the instance is trivially satisfiable). We set $k=3m+2n$, and construct a graph $G$ as follows.

For each variable $a_i$, we have a subgraph $F_i$ on nine vertices consisting of a cycle $r_i^+s_i^+s_i^-r_i^-t_i$ with two pendant paths $p_i^+q_i^+r_i^+$ and $p_i^-q_i^-r_i^-$.

For each clause $C_j$ we have a subgraph $H_j$ on five vertices consisting of a triangle $u_jv_jw_j$ and two pendant edges $v_jx_j$ and $w_jy_j$. All subgraphs $F_i$ and $H_j$ are disjoint.

For each pair $(i,j)$ such that $a_i$ or $\neg a_i$ appears in $C_j$, we connect $F_i$ and $H_j$ as follows. If $a_i$ appears, add an edge between $r_i^+$ and $u_j$, a path of length $2$ between $q_i^+$ and $w_j$, and a path of length $3$ between $r_i^-$ and $v_j$. If $\neg a_i$ appears, do the same with all signs reversed. All paths are internally vertex-disjoint both from each other and from the subgraphs previously constructed. Write $G_{i,j}$ for the subgraph consisting of $F_i$, $H_j$ and these connections between them. An example is shown in Figure \ref{fig:SAT}.

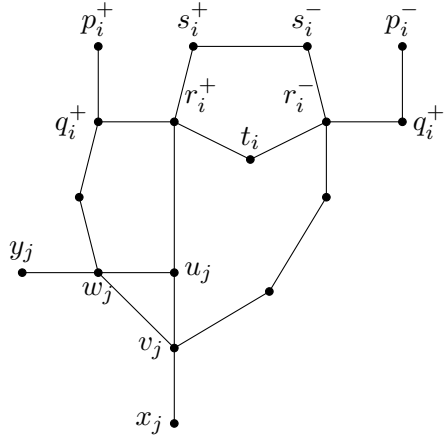
\begin{figure}
\centering
\begin{tikzpicture}
\filldraw (1,5) circle (0.05) node[anchor=south] {$p_i^+$};
\filldraw (2.25,5) circle (0.05) node[anchor=south] {$s_i^+$};
\filldraw (3.75,5) circle (0.05) node[anchor=south] {$s_i^-$};
\filldraw (5,5) circle (0.05) node[anchor=south] {$p_i^-$};
\filldraw (1,4) circle (0.05) node[anchor=east] {$q_i^+$};
\filldraw (2,4) circle (0.05) node[anchor=south west] {$r_i^+$};
\filldraw (4,4) circle (0.05) node[anchor=south east] {$r_i^-$};
\filldraw (5,4) circle (0.05) node[anchor=west] {$q_i^+$};
\filldraw (3,3.5) circle (0.05) node[anchor=south] {$t_i$};
\filldraw (0.75,3) circle (0.05);
\filldraw (4,3) circle (0.05);
\draw (1,5) -- (1,4) -- (2,4) -- (2.25,5) -- (3.75,5) -- (4,4) -- (5,4) -- (5,5);
\filldraw (0,2) circle (0.05) node[anchor=south] {$y_j$};
\filldraw (1,2) circle (0.05) node[anchor=north] {$w_j$};
\filldraw (2,2) circle (0.05) node[anchor=west] {$u_j$};
\filldraw (2,1) circle (0.05) node[anchor=east] {$v_j$};
\filldraw (2,0) circle (0.05) node[anchor=east] {$x_j$};
\filldraw (3.25,1.75) circle (0.05);
\draw (2,0) -- (2,1) -- (1,2) -- (0,2);
\draw (1,4) -- (0.75,3) -- (1,2) -- (2,2) -- (2,1) -- (3.25,1.75) -- (4,3) -- (4,4) -- (3,3.5) -- (2,4) -- (2,2);
\end{tikzpicture}
\caption{The subgraphs $F_i$ and $H_j$ and connections between them, where $a_i$ appears in $C_j$.}
\label{fig:SAT}
\end{figure}

Finally, add two new vertices $z_1,z_2$ and add all edges of the form $v_jz_1$ and $w_jz_2$.

Note that any MEG-set for this graph must contain $p^+_i$ and $p^-_i$ together with at least one of $s^+_i$ and $s_i^-$ for every $i$. The first two are necessary because they have degree $1$, and since every shortest path containing $s_i^+s_i^-$ starts or ends at one of those two vertices, this edge cannot be monitored if neither $s_i^+$ nor $s^-_i$ is in the set. Likewise, any MEG-set must contain $x_j$ and $y_j$ for each $j$.

We therefore consider sets consisting of $p_i^+$, $p_i^-$ and exactly one of $s_i^+$ or $s_i^-$ for each $i\leq m$, and $x_j$ and $y_j$ for each $j\leq n$. If $G$ has a MEG-set of order at most $k$, then it must be of this form. Sets of this form correspond to truth assignments, where $a_i$ is true if and only if $s_i^+$ is in the set. We will show that any set of this form monitors all edges except possibly edges of the form $u_jv_j$, and that $u_jv_j$ is monitored if and only if $C_j$ is satisfied in the corresponding truth assignment. Thus we will have $\meg{G}\leq k$ if and only if the instance is satisfiable. In the analysis that follows, we write $s_i^\pm$ to mean whichever of $s_i^+$ and $s_i^-$ is in the set, and we use e.g.\ $p_i^\pm$ and $q_i^\mp$ for vertices with the same or opposite sign respectively.

For each $i$, the pair $p_i^+,p_i^-$ monitors the edges of the path $p_i^+q_i^+r_i^+t_ir_i^-q_i^-p_i^-$, the pair $p_i^\pm,s_i^\pm$ monitors $r_i^\pm s_i^\pm$, and the pair $p_i^\mp s_i^\pm$ monitors $s_i^+s_i^-$ and $r_i^\mp s_i^\mp$.

For each $j$, the pair $x_j,y_j$ monitors the edges of the path $x_jv_jw_jy_j$.

Next we consider each pair $(i,j)$ such that $a_i$ or $\neg a_i$ appears in $C_j$. Suppose, without loss of generality, that $a_i$ appears in $C_j$ (as shown in Figure \ref{fig:SAT}). The pair $p_i^-,x_j$ is connected within $G_{i,j}$ by a path of length $6$. Any route which crosses into $F_{i'}$ or $H_{j'}$ for some $i'\neq i$ or $j'\neq j$ has length at least $7$. Thus this pair monitors the edges on the path of length $3$ between $r_i^-$ and $v_j$. Similarly, the pair $p_i^+, y_i$ monitors the path of length $2$ between $q_i^+$ and $w_j$. The pair $s_i^\pm,y_j$ monitors the edges $r_i^+u_j$ and $u_jw_j$, since there is a unique shortest path $s_i^+r_i^+u_jw_jy_j$ or $s_i^-s_i^+r_i^+u_jw_jy_j$ respectively. If $s_i^\pm=s_i^+$ then there is a unique shortest path of length $4$ between $s_i^\pm$ and $x_j$, which uses the edge $u_jv_j$, so this pair monitors that edge. However, if $s_i^\pm=s_i^-$ then there are two shortest paths of length $5$: $s_i^-s_i^+r_i^+u_jv_jx_j$ and another path via the path of length $3$ from $r_i^-$ to $v_j$, with the latter not using $u_jv_j$.

It follows that all edges other than those of the form $u_jv_j$, $x_jz_1$ or $y_jz_2$ are monitored by pairs thus far discussed, whereas $u_jv_j$ is monitored by $s_i^\pm,x_j$ if and only if the truth assignment for $a_i$ satisfies $C_j$. It remains to show that the other pairs, that is, pairs from $F_i$ and $F_{i'}$ for $i\neq i'$, or from $H_j$ and $H_{j'}$ for $j\neq j'$, or from $F_i$ and $H_j$ where neither $a_i$ nor $\neg a_i$ appears in $C_j$, monitor all edges of the form $x_jz_1$ or $y_jz_2$, but do not monitor $u_jv_j$ unless the clause $C_j$ is satisfied.

The pair $x_j,x_{j'}$ monitors the edges $x_jz_1$ and $x_{j'}z_1$, and likewise $y_j,y_{j'}$ monitors the edges $y_jz_2$ and $y_{j'}z_2$. There are two shortest paths between $x_j$ and $y_{j'}$, one via $v_{j'}$ and one via $w_{j}$, so this pair only monitors the edges $x_jv_j$ and $w_{j'}y_{j'}$.

We next consider shortest paths between pairs $(b,c)$ with $b\in\{p_i^+,s_i^\pm,p_i^-\}$ and $c\in\{p_{i'}^+,s_{i'}^\pm,p_{i'}^-\}$. Since each of $a_i$ and $\neg a_i$ appears somewhere, there is some $j$ with $\dist(b,w_j)=3$, and likewise there is some $j'$ with $\dist(c,w_{j'})=3$. This gives a path of length $8$ via $z_2$ (unless $j=j'$ in which case there is a path of length $6$). In order for a shorter path to exist and visit $v_{j''}$ for some $j''$, we must have $\dist(b,v_{j''}),\dist(c,v_{j''})\leq 4$. However, in that case we have $\dist(b,w_{j''})\leq\dist(b,v_{j''})$, and likewise for $c$, giving a shortest path avoiding $v_{j''}$. Consequently $b,c$ does not monitor any edge of the form $u_jv_j$.

Finally, we consider pairs $(b,c)$ with $b\in\{p_i^+,s_i^\pm,p_i^-\}$ and $c\in\{x_j,y_j\}$, where $a_i$ and $\neg a_i$ do not appear in $C_j$. Again, there is some $j'$ with $\dist(b,w_{j'})=3$, and so there is a path via $z_2$ of length $6$ if $c=y_j$ and $7$ if $c=x_j$ that does not use $u_jv_j$ or $u_{j'}v_{j'}$. The only way a shorter path can exist is if $c=x_j$ and $b=s_i^\pm$, in which case we may choose $j'$ such that $\dist(b,v_{j'})=3$ and obtain a path of length $6$ via $z_1$. While this path does use the edge $u_{j'}v_{j'}$, that edge is already monitored by $s_i^\pm,x_{j'}$, and so pairs of this form do not monitor any edges not monitored by other pairs already considered.

Thus $u_jv_j$ is monitored by some pair if and only if the truth assignment satisfies $C_j$, as required.
\end{proof}

\section*{Acknowledgements}
Research supported by the European Research Council under the European Union's Horizon 2020 research and innovation programme (grant agreement no.\ 883810).


\begin{thebibliography}{9}
\bibitem{BR03}Y. Bejerano and R. Rastogi,
\newblock Robust monitoring of link delays and faults in IP networks.
\newblock In \textit{Twenty-second Annual Joint Conference of the IEEE Computer and Communications Societies} (IEEE INFOCOM 2003), 134--144.

\bibitem{FKKMR}F. Foucaud, S. Kao, R. Klasing, M. Miller and J. Ryan,
\newblock Monitoring the edges of a graph using distances. 
\newblock\textit{Discrete Applied Mathematics} 319 (2022), 424--438.

\bibitem{FKMR} F. Foucaud, R. Klasing, M. Miller and J. Ryan,
\newblock Monitoring the edges of a graph using distances.
\newblock In \textit{Proceedings of the 6th International Conference on Algorithms and Discrete Applied Mathematics}
(CALDAM 2020), Lecture Notes in Computer Science 12016 (2020), 28--40.

\bibitem{FKL22+}F. Foucaud, N. Krishna and L. Ramasubramony Sulochana,
\newblock Monitoring edge-geodetic sets in graphs.
\newblock In \textit{Proceedings of the 9th International Conference on Algorithms and Discrete Applied Mathematics}
(CALDAM 2023), Lecture Notes in Computer Science 13947 (2023), 245--256.

\bibitem{HLT93} F. Harary, E. Loukakis, and C. Tsouros,
\newblock The geodetic number of a graph. 
\newblock\textit{Mathematical and Computational Modelling} 17 (1993), 89--95.

\bibitem{Kar72}R. M. Karp,
\newblock Reducibility among combinatorial problems.
\newblock In \textit{Complexity of computer computations} ({P}roc. {S}ympos., {IBM}
              {T}homas {J}. {W}atson {R}es. {C}enter, {Y}orktown {H}eights,
              {N}.{Y}., 1972), {85--103}.

\bibitem{MKXAT}P. Manuel, S. Klav\v{z}ar, A. Xavier, A. Arokiaraj and E. Thomas,
\newblock Strong edge geodetic problem in networks. 
\newblock\textit{Open Mathematics} 15:1 (2017), 1225--1235.
          
\bibitem{SJ07} A. P Santhakumaran, and J. John,
\newblock Edge geodetic number of a graph. 
\newblock\textit{Journal of Discrete Mathematical Sciences and Cryptography} 10 (2007), 415--432.
\end{thebibliography}
\end{document}